\documentclass[12pt,a4paper]{amsart}
\usepackage[margin=2cm]{geometry}
\usepackage{microtype}
\usepackage[cal=euler]{mathalfa}
\usepackage{tikz,tikz-cd}

\input{commands_new}
\usepackage[colorlinks,allcolors=blue]{hyperref}

\def\reg{\mathrm{reg}}
\def\s{\mathrm{s}}
\def\f{\mathsf{f}}
\def\br{\mathbf r}
\def\BB{Bia\l{}ynicki-Birula\xspace}
\def\idef{}

\begin{document}
\author{Sergey Mozgovoy}
\title{Motivic classes of Quot-schemes on surfaces}

\address{School of Mathematics, Trinity College Dublin, Ireland\newline\indent
Hamilton Mathematics Institute, Ireland}

\email{mozgovoy@maths.tcd.ie}

\begin{abstract}
Given a locally free coherent sheaf on a smooth algebraic surface, we consider the Quot-scheme parametrizing zero-dimensional quotients of the sheaf and find the corresponding motivic class in the Grothendieck ring of algebraic varieties.
\end{abstract}

\maketitle
\section{Introduction}
Let $X$ be a smooth algebraic variety and $E$ be a rank $r$ locally free coherent sheaf over~$X$.
For any $n\ge0$, let $\Quot(E,n)$ denote Grothendieck's \Quot-scheme \cite{grothendieck_techniques} parametrizing all epimorphisms $E\to F$, where $F$ is a zero-dimensional coherent sheaf with $\dim\Ga(X,F)=n$, modulo automorphisms of~$F$.
If~$X$ is projective, then so is $\Quot(E,n)$.

If $X$ is a curve, then $\Quot(E,n)$ is a smooth, connected variety of dimension $rn$.
The numbers of points of $\Quot(\cO_X^{\oplus r},n)$ over finite fields (as well as their Poincar\'e polynomials) were computed in~\cite{bifet_sur}
(see also \cite{ghione_effective,bifet_abel-jacobi}).
Motivic classes of $\Quot(\cO_X^{\oplus r},n)$ were computed in \cite{bano_chow} and motivic classes of $\Quot(E,n)$ for general $E$ were computed in \cite{bagnarol_motive}.
The general formula for motivic classes has the form
\begin{equation}\label{dim 1}
\sum_{n\ge0}[\Quot(E,n)]t^n=\Exp\rbr{[X\xx\bP^{r-1}]t},
\end{equation}
where $\Exp$ is the plethystic exponential (see \S\ref{mot classes}).

If $X$ is a surface, then $\Quot(E,n)$ is an irreducible variety of dimension $rn+n$ \cite{li_algebraic,gieseker_moduli,ellingsrud_irreducibility}.
The \Quot-scheme $\Quot(\cO_X,n)$ is the Hilbert scheme of points $\Hilb^n(X)$ which is smooth \cite{fogarty_algebraic}.
The numbers of points of $\Quot(\cO_X,n)$ over finite fields (as well as their Poincar\'e polynomials) were computed in \cite{gottsche_betti} and the numbers of points of $\Quot(\cO_X^{\oplus r},n)$ over finite fields were computed in \cite{yoshioka_betti}.
Motivic classes of $\Quot(\cO_X,n)$ were computed in \cite{gottsche_motive}.
In this paper we will prove the following result

\begin{theorem}
Let $E$ be a rank $r$ locally free sheaf over a smooth surface $X$. Then
$$\sum_{n\ge0}[\Quot(E,n)]t^n
=\Exp\rbr{\frac{[X\xx \bP^{r-1}]t}{1-\bL^rt}},$$
where $\bL=[\bA^1]$.
\end{theorem}

We will also give an alternative proof of the formula \eqref{dim 1}.
The idea of the computation goes back to \cite{gottsche_betti,yoshioka_betti}, where it was observed that in order to compute invariants of $\Quot(E,n)$ it is enough to compute invariants of the punctual quotient scheme $\Quot(E,n)_x$ consisting of quotients concentrated at one point $x\in X$.
The corresponding result in the motivic context was proved in~\cite{ricolfi_motive} (see also \cite{gusein-zade_power}).
The punctual quotient scheme can be interpreted as a quiver variety (see Proposition \ref{quiv-descript}).
In the case $\dim X=1$, this quiver variety is a nilpotent version of non-commutative Hilbert schemes studied in \cite{reineke_cohomology,engel_smooth}.
In the case $\dim X=2$, this quiver variety is a nilpotent version of Nakajima quiver varieties~ \cite{nakajima_instantons}.
In both cases the relevant quiver varieties admit a cellular decomposition and their motivic classes are well-understood.
It is an open problem to determine invariants of the punctual quotient schemes for $\dim X=3$,
although there is a work-around based on considering virtual motivic invariants (see \eg \cite{behrend_motivic}) and, essentially, reducing the question to the dimension two case.

In view of the above formula, we observe that for $X=\bA^2$ or $X=\bP^2$, the quotient scheme $\Quot(E,n)$ has a motivic class which is a polynomial in $\bL$ with non-negative coefficients.
This suggests that $Q=\Quot(E,n)$ should be pure (meaning that the cohomology groups $H^i_c(Q,\bC)$ are pure of weight $i$ for all $i\ge0$).
One could ask if, more generally, $\Quot(E,n)$ is pure for any smooth projective surface $X$.
Note that in this case $\Quot(\cO_X,n)$ is also smooth and projective, hence pure.

\section{Motivic classes and power structures}
\label{mot classes}

\sec[Motivic classes]
Let $\Var=\Var_k$ be the category of algebraic varieties over a field $k$ of characteristic zero.
Define the Grothendieck ring $K(\Var)$ of algebraic varieties over $k$ to be the free abelian group generated by isomorphism classes of objects in $\Var$ modulo the relations
$$[X]=[Y]+[X\ms Y]$$
for any variety $X$ and a closed subvariety $Y\sbs X$.
The ring structure is defined by $[X]\cdot [Y]=[X\xx Y]$, for $X,Y\in\Var$.
Let $K'(\Var)$ be the localization of $K(\Var)$ with respect to $\bL=[\bA^1]$.
The elements of $K(\Var)$ and $K'(\Var)$ will be called motivic classes.

\begin{remark}\label{pre-la}
Define a \idef{pre-\la-ring} structure on a commutative ring $R$ to be a group homomorphism
$$\si_t:(R,+)\to (1+tR\pser t,*),\qquad a\mto\si_t(a)= \sum_{n\ge0}\si_n(a)t^n,$$
such that $\si_t(a)=1+at+O(t^2)$.
The ring $K(\Var)$ is equipped with a pre-\la-ring structure 
$$\si_t([X])=\sum_{k\ge0}[S^kX]t^k,$$
where $S^kX=X^k/\fS_k$ is the $k$-th symmetric power,
for any quasi-projective variety $X$.
\end{remark}

There is an involutive ring homomorphism $K'(\Var)\to K'(\Var)$ \cite{mozgovoy_translation},
given by 
$$[X]\mto[X]\dual=\bL^{-\dim X}[X]$$
for every smooth, projective, connected variety $X$.
It satisfies $(\bL^n)\dual=\bL^{-n}$, for any $n\in\bZ$.

For $k=\bC$, define the virtual Poincar\'e polynomial
$$P:K'(\Var)\to\bZ[t^{\pm\oh}],\qquad 
[X]\mto\sum_{p,q,n}(-1)^n h^{p,q}(H^n_c(X,\bC))t^{\oh(p+q)},$$
where $h^{p,q}(H^n_c(X,\bC))$ is the dimension of the $(p,q)$-type Hodge component of the mixed Hodge structure on $H^n_c(X,\bC)$.
We have
$$P(X;t)=P([X];t)=\sum_n (-1)^n\dim H^n(X,\bC)t^{n/2}$$
for any smooth, projective variety $X$.
If the motivic class of an algebraic variety $X$ is a polynomial in $\bL$, then it coincides with the virtual Poincar\'e polynomial $P(X;\bL)$.
For any algebraic variety $X$, we have
$$P([X]\dual;t)=P([X];t\inv).$$

\sec[Power structures]
Define a \idef{power structure} over a commutative ring~$R$ to be a map
$$(1+tR\pser t)\xx R\to 1+tR\pser t,\qquad
(f,a)\mto f^a,$$
satisfying the following properties (\cf \cite{gusein-zade_power})
\begin{enumerate}
\item $f^0=1$,
\item $f^{a+b}=f^af^b$,
\item $f^1=f$,
\item \label{mult1} $f^{ab}=(f^a)^b$,
\item \label{mult2} $(fg)^a=f^ag^a$,
\item \label{PS:terms} $(1+t)^a=1+at+O(t^{2})$
\item \label{PS:power} $f(t^n)^a=f(t)^a|_{t\mto t^n}$, for all $n\ge1$,
\item \label{PS:cont}
it is continuous, meaning that for any $k\ge0$ there exists $n\ge0$ such that the $k$-jet of $f^a$ (\ie $f^a\pmod{t^{k+1}}$) is determined by the $n$-jet of $f$.
\end{enumerate}

\begin{remark}
A power structure over $R$ is uniquely determined by the pre-\la-ring structure on~$R$ given by 
$\si_t(a)=(1-t)^{-a}$ (see \cite{gusein-zade_power}).
Conversely, if a pre-\la-ring structure satisfies $\si_t(1)=\sum_{n\ge0}t^n$, then one can construct the corresponding power structure 
by first defining $(1-t^i)^{-a}=\sum_{n\ge0}\si_n(a)t^{in}$ and then defining
$$\rbr{\prod_{i\ge1}(1-t^i)^{-f_i}}^a
=\prod_{i\ge1}(1-t^i)^{-af_i}.$$
\end{remark}

There exists a power structure on the ring $K(\Var)$ (see \eg \cite{gusein-zade_powerb}) 
corresponding to the pre-\la-ring structure from Remark
\ref{pre-la}
$$(1-t)^{-[X]}=\si_t([X])=\sum_{k\ge0}[S^kX]t^{k}.$$

\sec[Plethystic exponentials]
Define a \idef{plethystic exponential} over a commutative ring $R$ to be a group homomorphism
$$\Exp:(tR\pser t,+)\to(1+tR\pser t,*)$$
satisfying the following properties
\begin{enumerate}
\item \label{E:tn} $\Exp(t)=(1-t)\inv$,
\item \label{E:terms} $\Exp(at)=1+at+O(t^{2})$,
\item \label{E:power} $\Exp(f(t^n))=\Exp(f(t))|_{t\mto t^n}$.
\item \label{E:cont} it is continuous, meaning that for every $k\ge0$ there exists $n\ge0$ such that the $k$-jet of $\Exp(f)$ is determined by the $n$-jet of $f$.
\end{enumerate}

\begin{remark}
There is a $1-1$ correspondence between pre-\la-ring structures on $R$ with $\si_t(1)=\sum_{n\ge0}t^n$ and plethystic exponentials, given by $\Exp\rbr{\sum_{i\ge1}f_it^i}=\prod_{i\ge1}\si_{t^i}(f_i)$.
\end{remark}

\medskip
Note that by continuity and the fact that $\Exp(at^n)=1+at^n+O(t^{n+1})$,
we have $$\Exp\rbr{\sum_{i\ge1}f_it^i}=\prod_{i\ge1}\Exp(f_it^i).$$
From this we conclude that \Exp is an isomorphism.
Let $\Log$ be its inverse map.

\begin{proposition}(\cf \cite{mozgovoy_computational})
\label{PS-Exp}
There is a $1-1$ correspondence between power structures and plethystic exponentials on $R$.
They are related by
$$\Exp\rbr{\sum_{n\ge1}f_n t^n}=\prod_{n\ge1}(1-t^n)^{-f_n},\qquad f^a=\Exp(a\Log(f)).$$
\end{proposition}
\begin{proof}
Consider a power structure on $R$.
One can show that $(1-t)^a=1-at+O(t^2)$,
hence $(1-t^n)^a=1-at^n+O(t^{n+1})$.
This implies that $\Exp$ given by the first formula is well-defined.
The axioms of an exponential follow from the axioms of a power structure.
Let $\Log$ denote the inverse of the map $\Exp$.
Given $f\in 1+tR\pser t$, let 
$\Log(f)=g=\sum_{n\ge1}g_nt^n$.
Then $f=\Exp(g)=\prod_{n\ge1}(1-t^n)^{-g_n}$ and by axioms \ref{mult1}, \ref{mult2}, \ref{PS:cont} we obtain
$$f^a
=\rbr{\prod_{n\ge1}(1-t^n)^{-g_n}}^{a}
=\prod_{n\ge1}(1-t^n)^{-g_n a}
=\Exp\rbr{\sum_{n\ge1} ag_nt^n}
=\Exp(ag).
$$
Conversely, given an exponential map $\Exp$, let $\Log$ be its inverse and let $f^a=\Exp(a\Log(f))$ for $f\in 1+tR\pser t$ and $a\in R$.
By continuity, for every $k\ge0$, there exists $n\ge k$ such that
$$\Exp\rbr{\sum_{i\ge1}f_it^i}
\equiv\Exp\rbr{\sum_{i=1}^n f_it^i}
=\prod_{i=1}^n\Exp(f_it^i)
\equiv\prod_{i=1}^k\Exp(f_it^i) \pmod{t^{k+1}}
,$$
hence we can take $n=k$.
Therefore
$$\Exp(at^n+O(t^{n+1}))\equiv \Exp(at^n)\equiv 1+at^n\pmod{t^{n+1}}.$$

Assume that $\Log(1+at^n+O(t^{n+1}))=bt^m+O(t^{m+1})$ for some $a,b\ne0$.
Then 
$$1+at^n+O(t^{n+1})=\Exp(bt^m+O(t^{m+1}))\equiv
1+bt^m \pmod{t^{m+1}}.$$
Therefore $m=n$ and $b=a$.
Every $g=\sum_{i\ge0}g_it^i\in1+tR\pser t$ can be written in the form $g=\rbr{\sum_{i=0}^n g_it^i}h$, where $h=1+O(t^{n+1})$.
Then 
$$\Log(g)=\Log\rbr{\sum_{i=0}^n g_it^i}+\Log(h)\equiv
\Log\rbr{\sum_{i=0}^n g_it^i} \pmod{t^{n+1}},$$
hence $\Log$ is continuous.
Therefore the power structure is continuous.
All other axioms of the power structure are easily verified.
Finally, we have $\Exp(t^n)=(1-t^n)\inv$, hence $\Log(1-t^n)=-t^n$ and, for any $f=\sum_{n\ge1}f_nt^n$, we obtain
$$\Exp(f)
=\prod_{n\ge1}\Exp(f_nt^n)
=\prod_{n\ge1}\Exp(-f_n\Log(1-t^n))
=\prod_{n\ge1}(1-t^n)^{-f_n}.
$$
\end{proof}

\begin{remark}
The plethystic exponential corresponding to the standard pre-\la-ring structure on $K(\Var)$ is given by
$$\Exp([X]t)=(1-t)^{-[X]}=\sum_{k\ge0}[S^kX]t^k$$
for any algebraic variety $X$
\end{remark}

\section{Nakajima quiver varieties}

\sec[Quiver varieties]
Let $Q=(Q_0,Q_1,s,t)$ be a finite quiver and let $kQ$ be its path algebra over a field $k$.
We define a $Q$-representation $M$ to be a pair $((M_i)_{i\in Q_0},(M_a)_{a\in Q_1})$, where $M_i$ is a vector space, for every $i\in Q_0$, and $M_a:M_i\to M_j$ is a linear map, for every arrow $a:i\to j$ in $Q$. 
We always assume that $\sum_{i\in Q_0}\dim M_i<\infty$ and identify $Q$-representations with finite-dimensional (left) $kQ$-modules.
For any path $u=a_n\dots a_1$, define $u|M=M_u=M_{a_n}\dots M_{a_1}$ considered as an endomorphism of the vector space $M=\bop_{i\in Q_0}M_i$.
Similarly, we have an endomorphism $u|M:M\to M$, for any element $u\in kQ$.

Let $A=kQ/I$, where $I\sbs kQ$ is an ideal contained in the ideal $J\sbs kQ$ generated by all arrows of $Q$.
The category $\mmod A$ of finite-dimensional, left $A$-modules can be identified with the category of $Q$-representations $M$ such that $u|M=0$, for all $u\in I$.
For any $M\in\mmod A$, we define its dimension vector 
$\udim M=(\dim M_i)_{i\in Q_0}\in\bN^{Q_0}$.
Given a vector $\te\in\bR^{Q_0}$, called a stability parameter, define the slope function
$$\mu_\te:\bN^{Q_0}\ms\set0\to\bR,\qquad \bv\mto\frac{\sum_i\te_i v_i}{\sum_i v_i},$$
and define $\mu_\te(M)=\mu_\te(\udim M)$, for $0\ne M\in\mmod A$.
An $A$-module $M$ is called $\te$-semistable (resp.\ $\te$-stable) if for any submodule $0\ne N\subsetneq M$, we have $\mu_\te(N)\le\mu_\te(M)$ (resp.\ $\mu_\te(N)<\mu_\te(M)$).

Let $V$  be a $Q_0$-graded vector space having dimension vector $\bv\in\bN^{Q_0}$.
The representation space $\cR(Q,\bv)=\bop_{a:i\to j}\Hom(V_i,V_j)$ is equipped with an action of the group $\GL_\bv=\prod_{i\in Q_0}\GL_{v_i}$ given by $(g\cdot M)_a=g_j M_a g_i\inv$, for $a:i\to j$ in $Q$.
For $A=kQ/I$, define
$$\cR(A,\bv)\sbs\cR(Q,\bv)$$
to be the closed subvariety consisting of representations that vanish on $I$.
There exists an open subvariety $\cR_\te(A,\bv)\sbs\cR(A,\bv)$ consisting of \te-semistable representations and an open subvariety $\cR_\te^\st(A,\bv)\sbs\cR_\te(A,\bv)$ consisting of \te-stable representations.
It is proved in \cite{king_moduli} that there exists a pre-projective categorical quotient
$$\cM_\te(A,\bv)=\cR_\te(A,\bv)\GIT\GL_\bv$$
that parametrizes $S$-equivalence classes of $\te$-semistable $A$-modules (in the category of $\te$-semistable $A$-modules having slope $\mu_\te(\bv)$).
There also exists a geometric quotient
$$\cM_\te^\st(A,\bv)=\cR_\te^\st(A,\bv)\qt\GL_\bv$$
which is open in $\cM_\te(A,\bv)$.
For the trivial stability $\te=0$, the moduli space
$$\cM_0(A,\bv)=\cR(A,\bv)\GIT\GL_\bv=\Spec k[\cR(A,\bv)]^{\GL_\bv}$$
parametrizes semi-simple $A$-modules having dimension vector $\bv$.
It is proved in \cite{king_moduli} that there exists a canonical projective morphism
$\pi:\cM_\te(A,\bv)\to\cM_0(A,\bv)$.
The fiber 
$$\cL_\te(A,\bv)=\pi\inv(0)$$
is projective
and parametrizes $S$-equivalence classes of nilpotent $\te$-semistable $A$-modules (here $M\in\mmod A$ is called nilpotent if $J^nM=0$ for some $n\ge1$).

\medskip
Let $\bd:Q_1\to\bZ$, $a\mto d_a$, be a map such that the ideal $I\sbs kQ$ is homogeneous with respect to the $\bZ$-grading on $kQ$ induced by $\bd$.
Then we can define the action of $T=\Gm$ on $\cM_\te(A,\bv)$
$$t\cdot M=(t^{d_a}M_a)_{a\in Q_1},\qquad t\in T,\, M\in\cM_\te(A,\bv).$$

\begin{proposition}[see \eg \cite{mozgovoy_translation}]
\label{prop:attractors}
Assume that $d_a>0$, for all $a\in Q_1$.
Then the action of $T$ on $\cM=\cM_\te(A,\bv)$ satisfies
\begin{enumerate}
\item 
$\cM^T$ is projective.
\item
$\cM^+=\sets{M\in\cM}{\exists\lim_{t\to0}tM}=\cM_\te(A,\bv)$.
\item
$\cM^-=\sets{M\in\cM}{\exists\lim_{t\to\infty}tM}=\cL_\te(A,\bv)$.
\end{enumerate}
\end{proposition}

This result implies (see \eg \cite{mozgovoy_translation})

\begin{proposition}
\label{L from M}
Assume that $\cM_\te(A,\bv)$ is  smooth and $d_a>0$, for all $a\in Q_1$. Then
$$[\cL_\te(A,\bv)]\dual=\bL^{-\dim\cM_\te(A,\bv)}[\cM_\te(A,\bv)].$$
\end{proposition}

\subsection{Nakajima quiver varieties}
Let $Q$ be a finite quiver and $\bw\in\bN^{Q_0}$ be a vector.
Define the framed quiver $Q^\f$ by adding to $Q$ one new vertex $*$ as well as $w_i$ arrows $*\to i$, for every $i\in Q_0$.
Define the double quiver $\bar Q^\f$ of the quiver $Q^\f$ by adding to $Q^\f$ an arrow $a^*:j\to i$, for every arrow $a:i\to j$ in $Q^\f$.
Define the pre-projective algebra
$$\Pi=k\bar Q^\f/(\fr),\qquad \fr=\sum_{(a:i\to j)\in Q^\f_1}(aa^*-a^*a).$$
Given $\bv\in\bN^{Q_0}$, we extend it to $\bv^\f\in\bN^{Q^\f_0}$ by setting $v^\f_*=1$.
Define a stability parameter $\te^\f\in\bR^{Q^\f_0}$ by setting $\te_*=1$ and $\te_i=0$ for $i\in Q_0$.
Define Nakajima quiver varieties \cite{nakajima_instantons}
$$\cM(\bv,\bw)=\cM_{\te^\f}(\Pi,\bv^\f),\qquad
\cL(\bv,\bw)=\cL_{\te^\f}(\Pi,\bv^\f).$$
Variety $\cM(\bv,\bw)$ is smooth and has dimension \cite{nakajima_instantons}
$$\dim \cM(\bv,\bw)=2(\bv\cdot\bw-\hi(\bv,\bv)),$$
where $\hi$ is the Euler-Ringel form of $Q$ defined by
$$\hi(\bv,\bw)=\sum_{i\in Q_0}v_iw_i-\sum_{(a\rcol i\to j)\in Q_1}v_iw_j,\qquad \bv,\bw\in\bZ^{Q_0}.$$
Actually, $\cM(\bv,\bw)$ is a symplectic manifold and $\cL(\bv,\bw)$ is its projective subvariety (Lagrangian if there are no loops in $Q$) homotopic to $\cM(\bv,\bw)$,
see \eg \cite{nakajima_instantons,ginzburg_lecturesa}.
Both varieties are pure.

\smallskip
It follows from the results of \cite{nakajima_quivere} (see also \cite{mozgovoy_translation})
that motivic classes of quiver varieties $\cM(\bv,\bw)$ and $\cL(\bv,\bw)$ are polynomials in $\bL$.
They are related by (see Proposition \ref{L from M})
\begin{equation}
\label{dual1}
[\cL(\bv,\bw)]\dual=\bL^{-\dim\cM(\bv,\bw)}[\cM(\bv,\bw)].
\end{equation}
There is an explicit formula for the motivic classes of $\cM(\bv,\bw)$ \cite{hausel_kac,mozgovoy_fermionic,wyss_motivic,bozec_number}
\begin{equation}\label{class1}
\sum_{\bv\in\bN^{Q_0}}\bL^{-\oh\dim\cM(\bv,\bw)}[\cM(\bv,\bw)]z^\bv
=\frac{\br(\bw,\bL,z)}{\br(0,\bL,z)},
\end{equation}
\begin{equation}
\br(\bw,q\inv,z)
=\sum_\ta q^{-\bw\cdot\ta_1}\prod_{k\ge1}q^{\hi(\ta_k,\ta_k)}\frac{z^{\ta_k}}{(q;q)_{\ta_k-\ta_{k+1}}},
\end{equation}
where
\begin{enumerate}
\item $\ta=(\ta^i)_{i\in Q_0}$ is a collection of partitions,
\item
$\ta_k=(\ta_k^i)_{i\in Q_0}\in\bN^{Q_0}$ for $k\ge1$,
\item
$z^\bv=\prod_{i\in Q_0}z_i^{v_i}$ for $\bv\in\bN^{Q_0}$,
\item $(t;q)_\bv=\prod_{i\in Q_0}(t;q)_{v_i}$, $(t;q)_n=\prod_{k=0}^{n-1}(1-tq^k)$ for $\bv\in\bN^{Q_0}$
and $n\in\bN$.
\end{enumerate}

\sec[Motivic classes for the Jordan quiver]
Consider the Jordan quiver $C^1$ which has one vertex and one loop.
Let $\cM(n,r)$ and $\cL(n,r)$ be the corresponding Nakajima quiver varieties ($r$ is the dimension of the framing).
The quiver variety $\cM(n,r)$ is smooth,
the quiver variety $\cL(n,r)$ is projective,
and their dimensions are (see below)
$$\dim\cM(n,r)=2rn,\qquad \cL(n,r)=rn-1.$$ 
According to \cite[\S3]{nakajima_lecturesa} there exists an action of $T=\Gm$ on $\cM=\cM(n,r)$ such that the fixed locus $\cM^T$ is finite and the attractors are 
$$\cM^+=\sets{M\in \cM}{\exists\lim_{t\to0}tM}=\cM(n,r),\qquad
\cM^-=\sets{M\in \cM}{\exists\lim_{t\to\infty}tM}=\cL(n,r).$$
By the \BB decomposition \cite{bialynicki-birula_some}, this implies that both varieties
have cellular decompositions.

\begin{theorem}
\label{motive of L,M}
We have
$$\sum_{n\ge0} [\cL(n,r)]t^n
=\prod_{i=1}^r\prod_{j\ge1}\frac1{1-\bL^{rj-i}t^j}
=\Exp\rbr{\frac{[\bP^{r-1}]t}{1-\bL^rt}},$$
$$\sum_{n\ge0} [\cM(n,r)]t^n
=\Exp\rbr{\frac{[\bP^{r-1}]\bL^{r+1}t}{1-\bL^rt}}.$$
\end{theorem}

\begin{proof}[First proof]
It is proved in \cite[Corollary 3.10]{nakajima_lecturesa} (by counting cells of a \BB decomposition) that the virtual Poincar\'e polynomials of $\cL(n,r)$ (which coincide with the usual Poincar\'e polynomials as $\cL(n,r)$ are pure and projective) satisfy
$$
\sum_n P(\cL(n,r),q)t^n
=\prod_{i=1}^r\prod_{j\ge1}\frac1{1-q^{rj-i}t^j}.
$$
As $\cL(n,r)$ admits a cellular decomposition, these polynomials also count motivic classes.
We have
$$
\prod_{i=1}^r\prod_{j\ge1}\frac1{1-\bL^{rj-i}t^j}
=\Exp\rbr{t\sum_{i=1}^r \bL^{r-i}\sum_{j\ge0}\bL^{rj}t^j}
=\Exp\rbr{\frac{\bL^r-1}{\bL-1}\frac{t}{1-\bL^rt}}.
$$
By the equation \eqref{dual1}, we have $[\cL(n,r)]\dual=\bL^{-2rn}[\cM(n,r)]$.
Therefore the first formula implies
$$\sum_{n\ge0} \bL^{-2rn}[\cM(n,r)]t^n
=\Exp\rbr{\frac{\bL^{-r}-1}{\bL\inv-1}\frac{t}{1-\bL^{-r}t}}
$$
and this is equivalent to the second formula.
\end{proof}

\begin{proof}[Second proof]
We will apply the general formula \eqref{class1} for the motivic classes of quiver varieties.
By the $q$-binomial theorem (Heine formula) we have
$$\sum_{n\ge0}\frac{t^n}{(q;q)_n}=\frac1{(t;q)_\infty}=\frac1{\prod_{k\ge0}(1-q^kt)}=\Exp\rbr{\frac t{1-q}},$$
where $(t;q)_n=\prod_{k=0}^{n-1}(1-q^kt)$.
For the quiver with one loop we have $\hi=0$, hence
\begin{multline*}
\br(r,q\inv,t)
=\sum_\ta q^{-r\ta_1}\prod_{k\ge1}\frac{t^{\ta_k}}{(q;q)_{\ta_k-\ta_{k+1}}}
=\sum_{m_1,m_2,\dots\ge0}
\prod_{k\ge1}\frac{q^{-rm_k}t^{km_k}}{(q;q)_{m_k}}\\
=\prod_{k\ge1}\rbr{\sum_{m\ge0}
\frac{(q^{-r}t^k)^m}{(q;q)_{m}}}
=\prod_{k\ge1}\Exp\rbr{\frac{q^{-r}t^k}{1-q}}
=\Exp\rbr{\frac{q^{-r}t}{(1-q)(1-t)}}
\end{multline*}
where we used
$m_k=\ta_{k}-\ta_{k+1}\ge0$ and $\ta_k=\sum_{i\ge k}m_i$ for $k\ge1$.
This implies
$$\sum_{n\ge0} \bL^{-rn}[\cM(n,r)]t^n
=\frac{\br(r,\bL,t)}{\br(0,\bL,t)}
=\Exp\rbr{\frac{\bL^{r}-1}{1-\bL\inv}\frac{t}{1-t}},
$$
hence the second formula of the theorem.
The first formula follows from the above argument.
\end{proof}

\begin{remark}
Note that the above formula for the motivic class of $\cL(n,r)$ implies that $\dim\cL(n,r)=rn-1$.
We will see later that $\cL(n,r)$ can be identified with the punctual scheme $\Quot(\cO^{\oplus r}_{\bA^2},n)_0$.
Its was proved in \cite{ellingsrud_irreducibility,baranovsky_punctual}
that this scheme is irreducible and has dimension $rn-1$.
\end{remark}

\sec[Relation to framed moduli spaces on \tpdf{$\bP^2$}{P2}]
Let $M(r,n)$ be the framed moduli space of torsion free sheaves on $\bP^2$ (see \eg 
\cite{nakajima_lectures,nakajima_lecturesa})
which parametrizes isomorphism classes of pairs $(E,\vi)$ such that
\begin{enumerate}
\item $E$ is a torsion free coherent sheaf on $\bP^2$, locally free in a \nbd of a line $\ell_\infty\sbs\bP^2$
ans satisfying $\rk E=r$, $c_2(E)=n$.
\item $\vi:E|_{\ell_\infty}\isoto \cO_{\ell_\infty}^{\oplus r}$ is an isomorphism, called framing.
\end{enumerate}

By a result of Barth \cite{barth_moduli} (see also \cite{nakajima_lectures})
this moduli space is isomorphic to the Nakajima quiver variety (note that we use the framing vector $\bw=r$ here)
$$M(r,n)\iso \cM(n,r)$$
for the quiver $C^1$ having one vertex and one loop.
There is a projective morphism
$$\pi:\cM(n,r)\to\cM_0(n,r)$$
which is an isomorphism over the moduli space $\cM_0^\s(n,r)$ of simple representations.
The preimage of $\cM_0^\s(n,r)$ corresponds to the moduli space $M_0^\reg(r,n)$ of framed locally free sheaves on~$\bP^2$ (identified by Donaldson \cite{donaldson_instantons} with the framed moduli space of instantons on~$S^4$).
Therefore we have (ADHM construction~ \cite{atiyah_construction})
$$M_0^\reg(r,n)\iso\cM_0^\s(n,r).$$

\section{Motivic classes of Quot-schemes}

\sec[Quot-schemes]
Let $E$ be a rank $r$ locally free sheaf over an algebraic variety $X$.
For any $n\ge0$, let $\Quot(E,n)$ denote the Grothendieck quotient scheme parametrizing epimorphisms $E\to F$, where $F$ is a zero-dimensional coherent sheaf with $\dim\Ga(X,F)=n$, modulo automorphisms of $F$.
For any point $x\in X$, let $\Quot(E,n)_x\sbs\Quot(E,n)$ denote the subscheme consisting of quotients $E\to F$ with $F$ supported in the point $x$.
This scheme depends only on a (formal) \nbd of $x\in X$.
If $X$ is smooth of dimension $d$, we have
$$\Quot(E,n)_x\iso\Quot(\cO_X^{\oplus r},n)_x\iso
\Quot(\cO_{\bA^d}^{\oplus r},n)_0.$$
The last scheme has a simple description as a nilpotent quiver variety.

\begin{proposition}
\label{quiv-descript}
Let $Q$ be a quiver with 
vertices $*$ and $1$, arrows $f_i:*\to 1$ for $1\le i\le r$, and loops $x_i:1\to 1$ for $1\le i\le d$
\def\temp{f_1,\dots,f_r}
\begin{ctikzcd}[execute at end picture={
		\draw[dotted](1.5,.9)to[out=0,in=90](2.2,.1);
		\draw[dotted](1.5,-.8)to[out=0,in=-90](2.2,-.1);
	}]
*\ar[rr,"\temp"]&&1
\ar[loop above,"x_1"]
\ar[loop below,"x_d"]
\ar[loop right,"x_i"]
\end{ctikzcd}
Let $A=kQ/(x_ix_j-x_jx_i)$, $\bv=(1,n)$,
$\te=(1,0)$
and let
$$\cM^d(n,r)=\cM_\te(A,\bv),\qquad
\cL^d(n,r)=\cL_\te(A,\bv)$$
be the corresponding quiver varieties.
Then 
$$\Quot(\cO^{\oplus r}_{\bA^d},n)\iso\cM^d(n,r),\qquad
\Quot(\cO_{\bA^d}^{\oplus r},n)_0\iso\cL^d(n,r).$$
\end{proposition}
\begin{proof}
Given a representation $M\in \cL_\te(A,\bv)$, the vector space $M_1$ is equipped with a module structure over $R=k[x_1,\dots,x_d]$ such that $x_i$ act nilpotently.
This implies that the corresponding coherent sheaf over $\bA^d$ is supported at $0$.
On the other hand we have $r$ linear maps $M_*=k\to M_1$ which induce a module homomorphism $R^{\oplus r}\to M_1$.
Stability condition means that $M_*$ generates representation $M$, hence the homomorphism $R^{\oplus r}\to M_1$ is surjective and we obtain a point in $\Quot(\cO_{\bA^d}^{\oplus r},n)_0$.
The converse correspondence is straightforward.
The proof for $\cM^d(n,r)$ is the same.
\end{proof}

\begin{theorem}[see \cite{ricolfi_motive}]
\label{quot-quot0}
Let $E$ be a rank $r$ locally free sheaf over a smooth algebraic variety $X$ of dimension~$d$. Then
$$\sum_{n\ge0}[\Quot(E,n)]t^n
=\rbr{\sum_{n\ge0}[\Quot(\cO_{\bA^d}^{\oplus r},n)_0]t^n}^{[X]}.$$
\end{theorem}

\sec[Quot-schemes over curves]
For $d=1$, the quiver $Q$ from Proposition \ref{quiv-descript} has the form
\def\temp{f_1,\dots,f_r}
\begin{ctikzcd}
*\ar[rr,"\temp"]&&1\ar[loop right,"x"]
\end{ctikzcd}
and has no relations.
The quiver variety $\cM^1(n,r)$ is smooth, has dimension $rn$ and admits a cellular decomposition (see \eg \cite{reineke_cohomology,engel_smooth}).

\begin{remark}
Let us show that $\cL^1(n,r)$ also admits a cellular decomposition.
There is a natural action of the torus $T'=\Gm^{Q_1}=\Gm^{r+1}$ on $\cM=\cM^1(n,r)$
such that the fixed locus $\cM^{T'}$ is finite.
We can find a torus $T=\Gm\sbs T'$ such that $\cM^{T}=\cM^{T'}$ and $T$ acts with a positive weight on every arrow.
The corresponding attractors are (see Proposition \ref{prop:attractors})
$$\cM^+=\sets{M\in\cM}{\exists\lim_{t\to0}tM}=\cM^1(n,r),
\qquad \cM^-=\sets{M\in\cM}{\exists\lim_{t\to\infty}tM}=\cL^1(n,r).$$
By the \BB decomposition \cite{bialynicki-birula_some}, this implies that both varieties have cellular decompositions.
\end{remark}

The virtual Poincar\'e polynomials of $\cM^1(n,r)$ (or equivalently, polynomials counting their points over finite fields) satisfy (see \eg \cite[\S5]{reineke_cohomology})
\begin{equation}\label{M1 PPoly}
\sum_{n\ge0}q^{-n}P(\cM^1(n,r),q)t^n
=\prod_{i=0}^{r-1}\frac1{1-q^it}
=\Exp\rbr{\frac{q^r-1}{q-1}t}.
\end{equation}

\begin{proposition}
\label{M-L1}
We have
$$\sum_{n\ge0}[\cM^1(n,r)]t^n
=\Exp\rbr{\bL\cdot [\bP^{r-1}] t},\qquad
\sum_{n\ge0}[\cL^1(n,r)]t^n
=\Exp\rbr{[\bP^{r-1}]t}.
$$
\end{proposition}
\begin{proof}
The first formula follows from equation \eqref{M1 PPoly} as $\cM^1(n,r)$ admits a cellular decomposition.
As $\cM^1(n,r)$ is smooth and has dimension $rn$, we conclude from Proposition \ref{L from M} that
$$[\cL^1(n,r)]\dual=\bL^{-rn}[\cM^1(n,r)].$$
Therefore
$$\sum_{n\ge0}[\cL^1(n,r)]\dual t^n
=\sum_{n\ge0}\bL^{-rn}[\cM^1(n,r)]t^n
=\Exp\rbr{\frac{\bL^r-1}{\bL-1}\bL^{1-r} t}
=\Exp\rbr{\frac{\bL^{-r}-1}{\bL\inv-1}t}$$
and taking the duals we obtain the second formula.
\end{proof}

\begin{remark}
We conclude from the above result that $\dim\cL^1(n,r)=rn-n$.
\end{remark}


\begin{theorem}[see \cite{bano_chow,bagnarol_motive}]
Given a rank $r$ locally free sheaf $E$ over a curve $X$, we have
$$\sum_{n\ge0}[\Quot(E,n)]t^n=\Exp\rbr{[X\xx\bP^{r-1}]t}.$$
\end{theorem}
\begin{proof}
By Proposition \ref{M-L1} we have
$\sum_{n\ge0}[\Quot(\cO_{\bA^1}^{\oplus r},n)_0]t^n
=\Exp\rbr{[\bP^{r-1}]t}$.
Applying Theorem \ref{quot-quot0} and
Proposition \ref{PS-Exp}
we obtain
$$\sum_{n\ge0}[\Quot(E,n)]t^n
=\Exp\rbr{[\bP^{r-1}]t}^{[X]}
=\Exp\rbr{[X][\bP^{r-1}]t}.
$$
\end{proof}

\begin{remark}
For $X=\bA^1$ and $E=\cO^{\oplus r}_{\bA^1}$, we have $\Quot(E,n)=\cM^1(n,r)$.
In this case we obtain
$$\sum_{n\ge0}[\cM^1(n,r)]t^n
=\sum_{n\ge0}[\Quot(E,n)]t^n
=\Exp\rbr{\bL\cdot [\bP^{r-1}]t}$$
which coincides with the statement of Proposition \ref{M-L1}.
\end{remark}

\begin{remark}
In the context of point-counting, plethystic exponential $\Exp([X]t)=\sum_{k\ge0}[S^kX]t^k$ corresponds to the zeta-function (for $X$ an algebraic variety over $\bF_q$)
$$Z(X;t)=\exp\rbr{\sum_{n\ge1}\frac{\#X(\bF_{q^n})}nt^n}.$$
Then the above theorem takes the form (\cf \cite{ghione_effective} for $E=\cO_X^{\oplus r}$)
$$\sum_{n\ge0}\#\Quot(E,n)t^n
=\prod_{i=0}^{r-1}Z(X\xx\bA^i;t)
=\prod_{i=0}^{r-1}Z(X;q^it).$$
\end{remark}

\sec[Quot-schemes over surfaces]
Given $n,r\ge0$, let $\cL^2(n,r)$ be
the quiver variety  from Proposition \ref{quiv-descript}.
On the other hand let $\cL(n,r)$ be the nilpotent Nakajima quiver variety for the quiver $C^1$ having one vertex and one loop.

\begin{proposition}
\label{L2 and L}
We have $\cL^2(n,r)\iso\cL(n,r)$.
\end{proposition}
\begin{proof}
The quiver variety $\cL(n,r)$ parametrizes representations of the quiver $\bar Q^\f$
having two vertices $*$ and $1$, two loops $x,x^*:1\to 1$
and arrows $f_i:*\to1$ and $f^*_i:1\to*$ for $1\le i\le r$.
The relations are
$$xx^*-x^*x+\sum_i f_if^*_i=0,\qquad \sum_i f_i^*f_i=0.$$
If $M\in\cL(n,r)$ then the linear maps $M_{f_i^*}:M_1\to M_*$ are zero (see \eg \cite{nakajima_instantons,lusztig_quiver}).
This implies that $M$ can be interpreted as a point of $\cL^2(n,r)$.
The converse is straightforward.
\end{proof}

\begin{proposition}
We have
$$\sum_{n\ge0}[\Quot(\cO_{\bA^2}^{\oplus r},n)_0]t^n
=\Exp\rbr{\frac{[\bP^{r-1}]t}{1-\bL^rt}}.$$
\end{proposition}
\begin{proof}
By Propositions \ref{quiv-descript} and \ref{L2 and L} we have
$$[\Quot(\cO_{\bA^2}^{\oplus r},n)_0]\iso\cL^2(n,r)\iso\cL(n,r).$$
Now the result follows from Theorem \ref{motive of L,M}.
\end{proof}

\begin{theorem}
Given a rank $r$ locally free sheaf $E$ over a smooth surface $X$, we have
$$\sum_{n\ge0}[\Quot(E,n)]t^n
=\Exp\rbr{\frac{[X\xx \bP^{r-1}]t}{1-\bL^rt}}.$$
\end{theorem}
\begin{proof}
By Theorem \ref{quot-quot0}
we have
$$\sum_{n\ge0}[\Quot(E,n)]t^n=
\rbr{\sum_{n\ge0}[\Quot(\cO_{\bA^2}^{\oplus r},n)_0]t^n}^{[X]}
=\Exp\rbr{\frac{[\bP^{r-1}]t}{1-\bL^rt}}^{[X]}.
$$
The last expression is equal to $\Exp\rbr{[X]\frac{[\bP^{r-1}]t}{1-\bL^rt}}$ by Proposition \ref{PS-Exp}.
\end{proof}

\begin{remark}
For $r=1$ we obtain (\cf \cite{gottsche_motive})
$$\sum_{n\ge0}[\Quot(\cO_X,n)]t^n
=\Exp\rbr{\frac{[X]t}{1-\bL t}}.$$
\end{remark}

\begin{remark}
We can write
$$\sum_{n\ge0}[\Quot(E,n)]t^n
=\Exp\rbr{\frac{[X\xx \bP^{r-1}]t}{1-\bL^rt}}
=\prod_{i=0}^{r-1}\prod_{j\ge0}\Exp([X]\bL^{i+rj}t^{j+1})
.$$
Therefore in the context of point-counting the above theorem, for an algebraic surface $X$ over a finite field $\bF_q$, takes the form (\cf \cite{yoshioka_betti} for $E=\cO_X^{\oplus r}$)
$$\sum_{n\ge0}\#\Quot(E,n)t^n
=\prod_{i=0}^{r-1}\prod_{j\ge0}Z(X;q^{i+rj}t^{j+1}).$$
\end{remark}

\begin{remark}
We have $\Quot(\cO^r_{\bA^2},n)\iso\cM^2(n,r)$, hence the above result implies
$$\sum_{n\ge0}[\cM^2(n,r)]t^n=\Exp\rbr{\frac{[\bP^{r-1}]\bL^2t]}{1-\bL^r t}}.$$
Note that $\cM^2(n,r)\sbs\cM(n,r)$ and 
$$\sum_{n\ge0} [\cM(n,r)]t^n
=\Exp\rbr{\frac{[\bP^{r-1}]\bL^{r+1}t}{1-\bL^rt}}$$
by Theorem \ref{motive of L,M}.
This implies that $\cM^2(n,r)$ and $\cM(n,r)$ are not equal in general.
However, we have an equality $\cM^2(n,1)=\cM(n,1)=\Hilb^n(\bA^2)$ (see \eg \cite{nakajima_lectures}).
\end{remark}


\providecommand{\bysame}{\leavevmode\hbox to3em{\hrulefill}\thinspace}
\providecommand{\href}[2]{#2}

\end{document}